\documentclass[a4paper,12pt,oneside]{amsart}

\usepackage[german,english]{babel}

\usepackage{amsmath,amssymb,amsthm,enumerate,mathtools}

\usepackage{dsfont}
\usepackage{color}
\usepackage[utf8]{inputenc}
\usepackage[automark,nouppercase]{scrpage2}
\usepackage{hyperref}
\usepackage[numbers]{natbib}
\usepackage{cancel}

\newcommand{\mE}{\mathbb{E}}

\newcommand{\mP}{\mathbb{P}}
\newcommand{\mR}{\mathbb{R}}

\newcommand{\cF}{\mathcal{F}}

\newcommand{\cG}{\mathcal{G}}

\newcommand{\lc}{\left\lbrace}
\newcommand{\rc}{\right\rbrace}
\newcommand{\la}{\left\lbrack}
\newcommand{\ra}{\right\rbrack}

\newcommand{\1}{\mathds{1}}
\newcommand\dd{\mathrm{d}}

\newtheorem {definition}{Definition}[section]
\newtheorem {theorem}[definition]{Theorem}

\newtheorem {lemma}[definition]{Lemma}

\newtheorem {remark}[definition]{Remark}

\DeclareMathOperator*{\sgn}{sgn}

\allowdisplaybreaks

\title[S\MakeLowercase{table processes conditioned to avoid an interval}]{Stable processes conditioned to avoid an interval}

\author{Leif D\"oring}
\address{Leif D\"oring: University of Mannheim, Institute of Mathematics, 68161 Mannheim, Germany.}
\email{doering@uni-mannheim.de}

\author{Andreas E. Kyprianou$^*$
}\thanks{$^*$Supported by EPSRC grants EP/L002442/1 and EP/M001784/1}
\address{Andreas E. Kyprianou: University of Bath, Department of Mathematical Sciences, Bath, BA2 7AY, UK.}
\email{a.kyprianou@bath.ac.uk}

\author{Philip Weißmann$^{**}$}\thanks{$^{**}$Supported by the Research Training Group "Statistical Modeling of Complex Systems"}
\address{Philip Weißmann: University of Mannheim, Institute of Mathematics, 68161 Mannheim, Germany.}
\email{hweissma@mail.uni-mannheim.de}

\begin{document}

\begin{abstract}
Conditioning Markov processes to avoid a domain is a classical problem that has been studied in many settings. Ingredients for standard arguments involve the leading order tail asymptotics of the distribution of the first hitting time of the domain of interest and its relation to  an underlying harmonic function. 

In the present article we condition stable processes to avoid intervals.  The required tail asymptotics in the stable setting for $\alpha\geq 1$ go back to classical work of Blumenthal et al. and  Port from the 1960s. For $\alpha<1$, we appeal to  recent results centred around the so-called deep  factorisation of the stable process to compute hitting probabilities and, moreover,  to identify the associated harmonic functions for all $\alpha\in (0,2)$. With these in hand, we thus prove that  conditioning to avoid an interval  is possible in the classical sense and that the resulting process is a Doob  $h$-transform of the stable process killed on entering the aforesaid interval. Appealing to the representation of the conditioned process as a Doob $h$-transform, we   verify that the conditioned process is transient.
\end{abstract}
\maketitle

\section{Introduction and main results} \label{intro}
Conditioning Markov processes to avoid sets is a classical problem. Indeed, suppose $P^x$, $x\in E$, represent a family of Markov probabilities on the state space $E$. Moreover, suppose that  $T$ is the first hitting time of a given domain. In the case that $T$ is $P^x$-almost surely finite, for each $x\in E$, it is non-trivial to construct and characterise the conditioned process through the natural limiting procedure
\begin{align}\label{cond}
	\lim_{s\to \infty}P^x(\Lambda\,|\, s+t<T),\quad \Lambda\in \mathcal F_t, x\in E,
\end{align}
where $\mathcal{F}_t$, $t\geq 0$ is the natural filtration of the underlying Markov process.

\smallskip

The most classical example in this respect is the Brownian motion conditioned to avoid the negative half-line; that is to say $P^x$, $x\in\mathbb{R}$, corresponds to the probabilities of Brownian motion and $T$ is the first hitting time of $(-\infty,0]$. See for instance Chapter VI.3 of Revuz and Yor \cite{Rev_Yor_01}. The conditioning \eqref{cond} can be made rigorous and is identified as the   the so-called Doob $h$-transform of the Brownian motion killed on exiting the upper half line. As such conditioned Brownian motion is uniquely characterised by the harmonic function $h(x)=x$ on $(0,\infty)$, from which it is an easy task to identify the conditioned process via its infinitesimal generator. Indeed,  the latter is equal to $\mathcal A^h f= \mathcal A (hf)/h$, where $f$ belongs to an appropriately rich set of test functions on $(0,\infty)$ and $\mathcal A$ is the infinitesimal  generator of Brownian motion killed on exiting $(0,\infty)$. It turns out that the resulting conditioned process is a  three-dimensional Bessel process. 

\smallskip

Extensions of this result have been obtained to condition L\'evy processes to stay positive, see Chaumont and Doney \cite{Chau_Don_01}. In that case, the associated  harmonic function that plays the role of  $h$ in the Brownian example above is given by the  the potential of the descending ladder height process. In a similar spirit Bertoin and Doney \cite{Bert_Don_01} have shown how to condition random walks conditioned to stay non-negative. Other examples of Markov processes conditioned to avoid domains via a limiting procedure, and thus characterised as a Doob-$h$ transform of the original process killed on exiting the specified domain, include  random walks conditioned to stay in cones (Denisov and Wachtel \cite{Den_Wac_01}),  random walks with finite second moments conditioned to avoid an interval (Vysotsky \cite{Vys_01}), spectrally negative L\'evy processes conditioned to stay in an interval (Lambert \cite{Lam_01}), subordinators conditioned to stay in an interval (Kyprianou et al. \cite{Kyp_Riv_Sen_02}), L\'evy processes conditioned to avoid the origin 
(Pant\'i  \cite{Pan_01} and Yano \cite{Yan_01}) or self-similar Markov processes conditioned to avoid the origin (Kyprianou et al. \cite{Kyp_Riv_Sat_01}).

\smallskip

 The purpose of this article is to take advantage of the path discontinuities of strictly stable processes ({\it henceforth referred to as  stable processes}) and to condition them to avoid an interval. Such a conditioning is not immediately obvious as, in contrast to L\'evy processes killed on the negative half-line, we cannot rely on the existence of a generic   positive harmonic function of a L\'evy processes killed in an interval.\smallskip
 





Let ${X} = ({X}_t)_{t \geq 0}$ be a stable  process and denote by $\mP^x$ the law of ${X}$ started from $x \in \mR$. More precisely, this means that ${X}$ is the canonical process on the space of c\`adl\`ag paths equipped with the $\sigma$-algebra $\cF$ induced by the Skorohod topology and $\mP^x$ is the probability measure that makes ${X}$ a stable process started from $x$. Stable processes observe the scaling property
\begin{align}\label{scaling}
\left((c {X}_{c^{-\alpha}t})_{t \geq 0},\mP^x \right) \overset{(d)}{=} \left(({X}_{t})_{t \geq 0},\mP^{cx}\right)
\end{align}
for all $x \in \mR$ and $c>0$, where $\alpha$ is the index of self-similarity. It turns out to be necessarily that case that $\alpha\in (0,2]$ with $\alpha=2$ corresponding to the Brownian motion. The continuity of sample paths excludes the Brownian motion from our study so we restrict to $\alpha\in (0,2)$.

\smallskip

As a L\'evy process, stable processes are characterised entirely by their jump measure  which is known to be
\begin{align*}
	\Pi(\dd x) = \frac{\Gamma(\alpha+1)}{\pi} \left\{
\frac{ \sin(\pi \alpha \rho) }{ x^{\alpha+1}} \1_{\{x > 0\}} + \frac{\sin(\pi \alpha \hat\rho)}{ {|x|}^{\alpha+1} }\1_{\{x < 0\}}
 \right\}\dd x,\qquad  x \in \mathbb{R},\end{align*}
where 
$\rho:=\mP^0({X}_1 \geq 0)$. For $\alpha \in (0,1)$ we exclude the case $\rho \in \lc 0,1 \rc$, in which case ${X}$ is (the negative of) a subordinator and for $\alpha \in (1,2)$, which forces $\rho\in[{1}/{\alpha}, 1- {1}/{\alpha} ]$, we exclude $\rho \in \left\lbrace {1}/{\alpha}, 1- {1}/{\alpha} \right\rbrace$, in which case ${X}$ has one-sided jumps.  For $\alpha=1$ the stable process reduces to the symmetric Cauchy process with $\rho=\frac 1 2$. We have taken a normalisation for which its characteristic exponent satisfies
\[
\mathbb{E}^x[{\rm e}^{{\rm i}\theta( {X}_1-x)}] = {\rm e}^{-|\theta|^\alpha}, \qquad \theta\in\mathbb{R}.
\]

An important fact we will use for the parameter regimes is that the stable process exhibits (set) transience and (set) recurrence according to whether $\alpha\in (0,1)$ or $\alpha\in[1,2)$. When $\alpha\in(1,2)$ the notion of recurrence is even stronger in the sense that individual points are hit with probability one. 
 \smallskip
 
  For an open or closed set $B \subseteq \mR$ let
$$T_{B} := \inf\left\lbrace t \geq 0: {X}_t \in  B \right\rbrace$$
be the first hitting time of $B$. It is known that this is a stopping time with respect to the natural enlargement of the filtration induced by ${X}$, which we denote by $(\mathcal{F}_t)_{t \geq 0}$ (i.e. the right-continuous extension, enlarged by null sets). For further details about stable processes we refer the reader to Chapter 7 of Bertoin \cite{Bert_01}, Chapter 3 of Sato \cite{Sat_01} or the review article Kyprianou \cite{ALEA_kyp}.

\smallskip

Stable processes with $\alpha\in [1,2)$ are recurrent, hence, the event $\left\lbrace T_{[a,b]}=+\infty \right\rbrace$ has zero probability. It is therefore non-trivial to establish that the limit
\begin{align}\label{cond_process}
	\lim\limits_{s \rightarrow \infty} \mP^x(\Lambda \,|\, t+ s <  T_{[a,b]})
\end{align}
exists for $\Lambda \in \mathcal{F}_t$, $x \in \mR \backslash [a,b]$, and is well defined in terms of characterising a conditioned process. We show that \eqref{cond_process} defines a probability measure on path space under which the canonical process is a Markov process on $\mR \backslash [a,b]$ with dynamics that are identified as a Doob $h$-transform of the original process killed in $[a,b]$. Let us quickly recall the concept of $h$-transforms. Denote by $p_t^{[a,b]}(x,\dd y) = \mP^x({X}_t \in \dd y, t < T_{[a,b]}), t\geq 0,$ the sub-Markov transition kernel of the process killed on first entry into $[a,b]$. Write  $({\texttt P}_t^{[a,b]})_{t\geq 0}$ for the  corresponding transition semigroup. A harmonic function for the killed process is a strictly positive, measurable function $h : \mR \backslash [a,b] \rightarrow (0,\infty)$ such that ${\texttt P}_t^{[a,b]}h(x) = h(x)$ or, equivalently,
\begin{align*}
	\mE^x \la \1_{\lc t < T_{[a,b]} \rc} h({X}_t) \ra = h(x),
\end{align*}
for all $x \in \mR \backslash [a,b]$ and all $t \geq 0$. Note that thanks to the Markov property this is equivalent to $(\1_{\lc t < T_{[a,b]} \rc} h({X}_t))_{t \geq 0}$ being a $\mP^x$-martingale with respect to $(\cF_t)_{t \geq 0}$. 
When $h$ is a positive harmonic function,  the associated Doob $h$-transform is defined via the change of measure
\begin{align}\label{def_htrafo}
	\mP^{\updownarrow,x}(\Lambda) \coloneqq \mE^x \la \1_\Lambda \1_{\lc t < T_{[a,b]} \rc} \frac{h({X}_t)}{h(x)}  \ra,\quad x\in\mR \backslash [a,b],
\end{align}
for $\Lambda \in \cF_t$. From Chapter 11 of Chung and Walsh \cite{Chu_Wal_01} we know that under $\mP^{\updownarrow,x}$ the canonical process is strong Markov and the transition semigroup satisfies 
\[
{\texttt P}_t^\updownarrow f(x)=\frac{1}{h(x)}{\texttt P}_t^{[a,b]} (fh)(x), \qquad x\in\mathbb{R}
\backslash[a,b],
\]
for bounded measurable $f$ on $\mathbb{R}
\backslash[a,b]$.

\smallskip

Before stating our main  results, note that we can easily reduce our analysis for the interval $[a,b]$ to the interval $[-1,1]$. Indeed, suppose that $h$ is a harmonic function for $({\texttt P}_t^{[-1,1]})_{t \geq 0}$. Then the function
$$h_{[a,b]}(x) \coloneqq h\left(\frac{2}{b-a}x - \frac{b+a}{b-a}\right), \quad x\in\mR \backslash [a,b],$$
is harmonic for $({\texttt P}_t^{[a,b]})_{t \geq 0}$. With the help of stationary independent increments and the scaling property for stable processes we have
\begin{align*}
\mE^{x} \la \1_{\lc t < T_{[a,b]} \rc} h\left(\frac{2}{b-a}{X}_t - \frac{b+a}{b-a}\right) \ra &= \mE^{x} \la \1_{\lc t < T_{[a,b]} \rc} h\left(\frac{2}{b-a}\left({X}_t - \frac{b+a}{2}\right)\right) \ra \\
&= \mE^{x - \frac{b+a}{2}} \la \1_{\lc t < T_{[-\frac{b-a}{2},\frac{b-a}{2}]} \rc} h\left(\frac{2}{b-a}{X}_t\right) \ra \\
&= \mE^{\frac{2}{b-a}(x - \frac{b+a}{2})} \la \1_{\lc \left(\frac{b-a}{2}\right)^{-\alpha} t <  T_{[-1,1]} \rc} h\left({X}_{\left(\frac{b-a}{2}\right)^{-\alpha}t}\right) \ra \\
&= h\left(\frac{2}{b-a}x - \frac{b+a}{b-a}\right).
\end{align*}
As a consequence we will focus for the rest of the article on the case $[a,b] = [-1,1]$.

\smallskip

As a prelude to our first theorem, let us introduce the function
\begin{align*}
	\psi_{\alpha\rho}(z) = (z-1)^{\alpha\hat{\rho}-1}(z+1)^{\alpha\rho-1},\quad z\geq 1,
\end{align*}
and the analogous expression $\psi_{\alpha\hat{\rho}}$, for which $\rho$ is replaced with with $\hat{\rho} \coloneqq 1-\rho$. The first result identifies positive harmonic functions for the stable processes killed on entering $[-1,1]$.
\begin{theorem} \label{thm_harmonic}
Let ${X}$ be a two-side stable process with $\alpha \in (0,2)$, then
$$h(x) \coloneqq
\begin{cases}
 \int_1^x \psi_{\alpha\rho}(z) \, \dd z&\text{if }  x >1 \\
 \int_1^{|x|} \psi_{\alpha\hat{\rho}}(z) \, \dd z&\text{if }  x <-1
\end{cases}, $$
is a harmonic function for $({\emph{\texttt P}}_t^{[-1,1]})_{t \geq 0}$.
\end{theorem}
In the case $\alpha=1$ we can write the above harmonic function in explicit detail:
\begin{align}\label{precise}
	h(x) =  \int_1^{|x|} (z^2-1)^{-\frac{1}{2}} \, \dd z = \log\left(|x| + (x^2-1)^{\frac{1}{2}}\right),\qquad |x|>1.
\end{align}
The next result addresses  our main motivation, namely to give a precise meaning to conditioning the stable processes to avoid an interval and to characterise the resulting process. Combining the harmonic functions with classical results of Blumenthal et al. \cite{Blu_Get_Ray_01} and Port \cite{Por_01} we can prove that the usual conditioning procedure works and at the same time identify the conditioned processes as $h$-transforms of the killed processes with the harmonic functions from Theorem \ref{thm_harmonic}.
\begin{theorem}\label{thm_cond}
Let ${X}$ be a two-side stable process with $\alpha \in (0,2)$, then
\begin{align*}
	\lim_{s \rightarrow \infty} \mP^x(\Lambda \,|\, t+s < T_{[-1,1]}) = \mP^{\updownarrow,x}(\Lambda) := \mE^x \la \1_\Lambda \1_{\lc t < T_{[-1,1]} \rc} \frac{h({X}_t)}{h(x)}  \ra \end{align*}
for $t\geq 0$, $x\notin [-1,1]$ and $\Lambda \in \cF_t$.
\end{theorem}
The reader will note  that conditioning to avoid an interval is  much simpler when  $\alpha<1$ thanks to transience of the stable process. Indeed, the conditioning is a conditioning on  a positive probability event. Moreover,   the probability is exactly proportional to the harmonic function $h$ from \eqref{thm_harmonic}.

\smallskip

As part of our characterisation of the conditioned process  we show that it is transient. This is in analogy to L\'evy processes conditioned to be positive which almost surely drift to infinity, see Chaumont and Doney \cite{Chau_Don_01}.
\begin{theorem}\label{thm_trans}
Let ${X}$ be  stable process with both positive and negative jumps  with $\alpha \in (0,2)$, then the conditioned process is transient in the sense that
\begin{align*}
	\mP^{\updownarrow,x}\Bigg( \int_0^\infty \1_{\lc {X}_t \in K \rc} \, \dd t < +\infty\Bigg)=1,\qquad |x|>1,
\end{align*}
for all compact subsets $K$ of $\mR \backslash [-1,1]$.
\end{theorem}

\section{Preliminaries}\label{sec:pre}
The methods to prove our two main theorems differ for the three parameter regimes $\alpha\in (0,1)$, $\alpha=1$ and $\alpha\in (1,2)$, where different background machinery will be employed.\smallskip

\textbf{Techniques for $\alpha\in (0,1)$:} Transience for $\alpha<1$ implies that $\lc T_{[-1,1]} = +\infty \rc$ has strictly positive probability, irrespective of the point of issue. We apply a recent formula from Kyprianou et al. \cite{Kyp_Riv_Sen_01}, which gives the  the law of the  point of closest reach of $0$ of the stable process. From the law of the point of closest reach we can precisely compute   $h$ which, up to a multiplicative constant, is equal to $\mathbb{P}^x(T_{[-1,1]} = +\infty)$, $x\in\mathbb{R}\backslash [-1,1]$. Since here we condition on a positive probability event the limit theorem is simple.
\smallskip

\textbf{Techniques for $\alpha=1$:} We employ the fact that stable processes are self-similar Markov processes and self-similar Markov processes can be represented as time-space transformation of Markov additive processes (MAPs). The so-called Lamperti-Kiu representation has been subject of intensive research in the past few years.\smallskip

 A Markov process $X=(X_t)_{t \geq 0}$ killed when first hitting $0$ is called a self-similar Markov process (ssMp) if
$$((c X_{c^{-\alpha}t})_{t \geq 0},\mP^x) \overset{(d)}{=} ((X_{t})_{t \geq 0},\mP^{cx})$$
for all $c>0$ and all $x \in \mR$, where $\alpha$ is the index of self-similarlity. A Markov additive process (MAP) with respect to a filtration $(\cG_t)_{t \geq 0}$ is a $\mR\times E$-valued c\`adl\`ag process $(\xi,J) = (\xi_t,J_t)_{t \geq 0}$, where $E$ is a finite space, $J$ is a Markov chain in $E$ and the following holds for all $i \in E$ and $s,t \geq 0$: given $\{J_t = i\}$, the pair $(\xi_{t+s}-\xi_t,J_{t+s})$ is independent of $\mathcal{G}_t$ and has the same distribution as $(\xi_{s}-\xi_0,J_{s})$ given $\lc J_0 = i \rc$. Asmussen \cite{Asm_01} provides the following alternative characterisation of MAPs in terms of L\'evy  processes, which helps us to get a clear image of the behavior of MAPs. The pair $(\xi,J)$ is a MAP if and only if for all $i \in E$ there is an iid sequence $(\xi^{n,i})_{n \in \mathbb{N}}$ of L\'evy  processes and for all $i,j \in E$ there is an iid sequence of random variables $(U^n_{i,j})_{n \in \mathbb{N}}$, independent of $J$ such that for the jump times $(\sigma_n)_{n \in \mathbb{N}}$ of $J$ and the convention $\sigma_0=0$ it holds that
$$\xi_t = \1_{\left\lbrace n \geq 1 \right\rbrace} \xi_{\sigma_n-} + U^n_{J_{\sigma_n}-,J_{\sigma_n}} +  \xi^{n,J_{\sigma_n}}_{t-\sigma_n}, \qquad t \in [\sigma_n,\sigma_{n+1}),\, n \geq 0.$$
In words, this characterisation shows that the dynamics of a MAP are governed by a set of L\'evy processes with different characteristics and the choice of the L\'evy process is governed by an underlying Markov chain.\smallskip

 Given a self-similar Markov process $X$ which starts from $x \in \mR \backslash \lc 0 \rc$, Chaumont et al. \cite{Chau_Pan_Riv_01} proved that there exists a MAP $(\xi,J)$ on $\mR \times \lc \pm 1 \rc$ which starts from $(\log|x|,\sgn(x))$ such that
$$X_t = J_{\varphi_t} \exp\left(\xi_{\varphi_t} \right),\qquad 0 \leq t < T_0,$$
where $\varphi_t = \inf\{ s>0: \int_0^s \exp(\alpha \xi_u) \, \dd u > t \}$ and $T_0 = \inf\left\lbrace t \geq 0 : X_t = 0 \right\rbrace$.
The transformation offers the opportunity to  obtain results for ssMps from the reservoir of proved and potentially provable results
for MAPs as long as the time-change can be controlled. As an example, the 0-1 law
for the longtime behavior of a MAP readily implies a 0-1 law for the extinction of a
ssMp. That is:
\begin{enumerate}[(i)]
\item $T_0 = +\infty$ a.s. if and only if the underlying MAP $(\xi,J)$ oscillates or drifts to $+\infty$.
\item $T_0 < +\infty$ a.s. if and only if the underlying MAP $(\xi,J)$ drifts to $-\infty$ or is killed.
\end{enumerate}

\smallskip

In the appendix of Dereich et al. \cite{Der_Doer_Kyp_01} the authors analysed questions on fluctuation theory for MAPs. In particular, analogously to L\'evy processes, it was proved that the potential function $U^-$ of the ascending ladder height process (see Section \ref{a=1} below) of the so-called dual process to the MAP is harmonic for the MAP killed on entering the negative half-line when it oscillates or drifts to $+\infty$. 

\smallskip

Since stable processes are ssMps, the Lamperti-Kiu transform is applicable to the stable process ${X}$. The underlying MAP was characterised in Kyprianou \cite{Kyp_02} and Kyprianou et al. \cite{Kyp_Riv_Sen_01}. For $\alpha=1$ it holds that $T_0 = +\infty$ almost surely (which is a consequence of the fact that all points  are polar for the Cauchy process), hence, the underlying MAP oscillates or drifts to $+\infty$ (in fact the former is the case). In \cite{Kyp_Riv_Sen_01} the authors were able to calculate explicit densities for the aforementioned harmonic function of the MAP underlying the Cauchy process. Since the Lamperti-Kiu transform tells us that killing the MAP on the negative half-line is equivalent to killing the stable process in $[-1,1]$, an appropriate spatial transform of this MAP turns out to be the  the key to prove Theorem \ref{thm_harmonic} in the Cauchy setting. 

\smallskip


\textbf{Techniques for $\alpha\in (1,2)$:} For $\alpha>1$ the stable process visits points almost surely, irrespective of its point of issue. We consider the stable process killed on hitting $0$. From Kyprianou et al. \cite{Kyp_Riv_Sat_01} (for general self-similar Markov processes) or Pantí \cite{Pan_01} and Yano \cite{Yan_01} (for general L\'evy  processes) we know that $ |x|^{\alpha-1}$, $x \neq 0$, is harmonic for the killed transition semigroup
$${\texttt P}^{\dag}_t f(x) := \mE^{\dag,x} \la f({X}_t) \ra := \mE^x \la f({X}_t) \1_{\lc t \leq T_0\rc} \ra,$$
where $T_0 := \inf\lc t \geq 0: {X}_t =0 \rc$.  Moreover, its corresponding Doob $h$-transform with $|x|^{\alpha-1}$, $x\neq0$, corresponds to  the stable process conditioned to avoid $0$ in a sense that also conforms to the general notion highlighted in  \eqref{cond}. We will show that $h^\circ(x): =  |x|^{1-\alpha} h(x), |x|>1,$ with $h$ defined in Theorem \ref{thm_harmonic}, is harmonic for the stable process conditioned to avoid the origin and killed on entering $[-1,1]$. In that case, by the compounding effect of Doob $h$-transforms,  it must be the case that  $h(x)=|x|^{\alpha-1}h^\circ(x)$ is harmonic for ${X}$ killed on entering $[-1,1]$. The tool we use to show harmonicity of $h^\circ(x)$ for the stable process conditioned to avoid $0$ and killed on entering $[-1,1]$ is the Riesz--Bogdan--\.Zak transform. To state the transformation let
$$\mP^{\circ,x}(\Lambda) \coloneqq \mE^x \la \1_\Lambda \1_{\lc t < T_{\lc 0 \rc} \rc} \frac{|{X}_t|^{\alpha-1}}{|x|^{\alpha-1}} \ra, \quad \Lambda \in \cF_t,
t\geq 0,
$$
be the law of the stable process conditioned to avoid $0$. The next Theorem was proved in Bogdan and \.Zak \cite{Bog_Zak_01} for symmetric stable processes, in Kyprianou \cite{Kyp_02} for general stable processes and in Alili et al. \cite{Ali_Cha_Gra_Zak_01} for self-similar Markov processes.
\begin{theorem}[Riesz--Bogdan--\.Zak transform] 
Let ${X}$ be a two-side stable process with $\alpha \in (0,2)$ and
$$\eta_t = \inf \Big\{ s>0 :  \int_0^s |{X}_u|^{-2\alpha} \1_{\lc u < T_{ 0 } \rc} \, \dd u > t \Big\},\qquad t < \int_0^\infty |{X}_u|^{-2\alpha} \1_{\lc u < T_{ 0} \rc} \, \dd u.$$
Then, for all $x \neq 0$, it holds that
$$\Big(\Big(\frac{1}{{X}_{\eta_t}}\Big)_{t \geq 0}, \hat{\mP}^x\Big) \overset{(d)}{=} ({X},\mP^{\circ,\frac{1}{x}}),$$
where $\hat{\mP}^x$ is the law of the dual stable process $-{X}$ issued from $x$.
\end{theorem}
The Riesz--Bogdan--\.Zak transformation tells that, for $\alpha \geq 1$, the stable process conditioned to avoid $0$ is the spatial inverse of the dual process with a certain time-change. We will use this theorem to show that $\mP^{\circ,x}(T_{[-1,1]} = +\infty) >0$ and to calculate this probability explicitly using a very recent formula for the distribution of the point of furthest reach of a stable process prior to hitting the origin.

\section{Proofs}
\subsection{Proof of Theorem \ref{thm_harmonic}}
The proofs for the cases $\alpha\in (0,1)$, $\alpha=1$ and $\alpha\in (1,2)$ are completely different in nature but all rely on recent new explicit formulas for stable processes obtained through the Lamperti-Kiu representation.
\subsubsection{$\underline{\alpha \in (0,1)}$}\label{a<1}
Since ${X}$ is transient, $\mP^x(T_{[-1,1]} = \infty)>0$ for all $x \notin [-1,1]$ and the function
\begin{align*}
	g(x)= \mP^x(T_{[-1,1]} = \infty),\quad x\in \mR\backslash [-1,1],
\end{align*}
is strictly positive. The following standard argument shows that $g$ is a harmonic function for the process killed on entering $[-1,1]$:
\begin{align*}
\mE^x \la \1_{\lc t < T_{[-1,1]} \rc} g({X}_t) \ra &= \mE^x \la \1_{\lc t < T_{[-1,1]} \rc} \mP^{{X}_t}(T_{[-1,1]} = \infty) \ra \\
&= \lim_{s \rightarrow \infty} \mE^x \la \1_{\lc t < T_{[-1,1]} \rc} \mP^{{X}_t}(T_{[-1,1]} > s) \ra \\
&= \lim_{s \rightarrow \infty} \mP^x(T_{[-1,1]} > s+t) \\
&= \mP^{x}(T_{[-1,1]} = \infty) \\
&= g(x),
\end{align*}
where we used dominated convergence in the second equation and the Markov property in the third. It remains to show that $h(x)$ equals $\mP^x(T_{[-1,1]} = \infty)$ up to some multiplicative constant. Proposition 1.1 of Kyprianou et al. \cite{Kyp_Riv_Sen_01} gives us the distribution of the point of closest reach to $0$ of the stable process from which the probability of missing $[-1,1]$ can be computed readily. Define $\underline{m}$ as the unique time such that $|{X}_t| \geq |{X}_{\underline{m}}|$ for all $t \geq 0$ and let $x>1$, then using the aforesaid result in \cite{Kyp_Riv_Sen_01},
\begin{align*}
\mP^x(T_{[-1,1]} = \infty) &= \mP^x({X}_{\underline{m}} > 1) + \mP^x({X}_{\underline{m}} < -1)\\
&= \frac{\Gamma(1-\alpha\rho)}{\Gamma(1-\alpha)\Gamma(\alpha\hat{\rho})} \left( \int_1^x \frac{x+z}{(2z)^\alpha}(x-z)^{\alpha\hat{\rho}-1} (x+z)^{\alpha\rho-1} \, \dd z \right. \\
&\quad \left. \hspace{3cm}+  \int_{-x}^{-1} \frac{x+z}{(-2z)^\alpha}(x+z)^{\alpha\hat{\rho}-1} (x-z)^{\alpha\rho-1} \, \dd z \right)\\
&= \frac{2\Gamma(1-\alpha\rho)}{2^\alpha\Gamma(1-\alpha)\Gamma(\alpha\hat{\rho})} x  \int_1^x \frac{1}{z^2}\left(\frac{x}{z}-1\right)^{\alpha\hat{\rho}-1}\left(\frac{x}{z}+1\right)^{\alpha\rho-1} \, \dd z\\
&= \frac{2^{1-\alpha}\Gamma(1-\alpha\rho)}{\Gamma(1-\alpha)\Gamma(\alpha\hat{\rho})}  \int_1^x \left(u-1\right)^{\alpha\hat{\rho}-1}\left(u+1\right)^{\alpha\rho-1} \, \dd u\\
&= \frac{2^{1-\alpha}\Gamma(1-\alpha\rho)}{\Gamma(1-\alpha)\Gamma(\alpha\hat{\rho})}  \int_1^x \psi_{\alpha\rho}(u) \, \dd u\\
&= \frac{2^{1-\alpha}\Gamma(1-\alpha\rho)}{\Gamma(1-\alpha)\Gamma(\alpha\hat{\rho})} h(x),
\end{align*}
where in the third equality we have substituted $u = x/z$.
If $x<-1$ we apply duality to deduce
$$\mP^x({X}_{\underline{m}} \in \dd y) = \hat{\mP}^{-x}({X}_{\underline{m}} \in -\dd y),$$
from which an analogous calculation for $x<-1$ yields the claim.

\subsubsection{$\underline{\alpha = 1}$}\label{a=1}
Let $(\mathbf{P}^{x,i})_{x \in \mR, i \in \lc \pm 1 \rc}$ be the family of probability measures on the space of $\mR \times \lc \pm 1 \rc$-valued c\`adl\`ag paths under which the canonical process $(\xi,J)$ has the distribution of the Markov additive process (MAP) which underlies the stable process (seen as an ssMp) via the Lamperti-Kiu transform. More precisely, this means that, under $\mathbf{P}^{\log |x|,\sgn(x)}$, the transformation $(J_{\varphi_t} \exp(\xi_{\varphi_t}))_{t\geq 0}$ is a ssMp with law $\mP^x$, i.e. the stable process. 

\smallskip

We will need to introduce some terminology from  Dereich et al. \cite{Der_Doer_Kyp_01} in order to talk about the ladder height processes of $(\xi, J)$. To this end, let  $Y_t =\xi_t- \underline{\xi}_t$, $t\geq 0$, where  $\underline{\xi}_t = \inf_{s\leq t}\xi_s$. 
Following ideas that are well known from the theory of L\'evy processes, it is straightforward to show that, as a pair, the process $(Y,J)$ is a strong Markov process. It was shown in the Appendix of Dereich et al. \cite{Der_Doer_Kyp_01} that there exists a local time of $(Y, J)$ in the set $\{0\}\times\{-1,1\}$, say $L: = (L_t)_{t\geq 0}$. Moreover, the process  $({L}^{-1}, H^-, J^-): = ({L}^{-1}_t, H^-_t, J^-_t)_{t\geq 0}$ is a (possibly killed)  Markov additive bivariate subordinator (meaning that it is a possibly killed MAP for which the components $({L}^{-1}_t)_{t\geq 0}$ and $(H^-_t)_{t\geq 0}$ are increasing), where	
\[
H^-_t : = \xi_{{L}^{-1}_t}\text{ and } J^-_t : = J_{{L}^{-1}_t}, \qquad \text{ if } {L}^{-1}_t<\infty,
\]
and $H^-_t : = \Delta$ and $J^-_t : = \Delta$ otherwise, for some cemetery state $\Delta$.

\smallskip

The process $(H^-,J^-)$ is called the descending ladder MAP. Next define
$$U^-_{i,j}(x) = \mathbf{E}^{0,i} \left\lbrack  \int_0^\infty \1_{\left\lbrace H^-_t \leq x,J^-_t = j \right\rbrace} \, \dd t \right\rbrack$$
which we call potential function of $(H^-,J^-)$ and
$$U^-_i(x) := U^-_{i,-1}(x) +  U^-_{i,1}(x) = \mathbf{E}^{0,i} \left\lbrack  \int_0^\infty \1_{\left\lbrace H^-_t \leq x \right\rbrace} \, \dd t \right\rbrack.$$
In the case $\alpha=1$ the stable process ${X}$ does not hit points, i.e. $T_0 = +\infty$ almost surely. It follows from the Lamperti-Kiu transfrom that the underlying MAP oscillates or drifts to $+\infty$ and in this case the function
$$(x,i) \mapsto U^-_i(x)$$
is harmonic for the MAP killed on entering the negative half-line, i.e.
$$\mathbf{E}^{x,i} \la \1_{\lc t < \tau_{(-\infty,0]} \rc} U^-_{J_t}(\xi_t) \ra = U^-_{i}(x)$$
for all $x >0$, where $\tau_{(-\infty,0]} \coloneqq \inf\lc t \geq 0: \xi_t \leq 0 \rc$ (see Theorem 29 of Dereich et al. \cite{Der_Doer_Kyp_01}). Using this we see
\begin{align*}
&\quad \mE^{x} \la \1_{\lc t < T_{[-1,1]} \rc} U^-_{\sgn({X}_t)}(\log |{X}_t|) \ra \\
=& \, \mathbf{E}^{\log |x|,\sgn(x)} \la \1_{\lc \varphi_t < \tau_{(-\infty,0]} \rc} U^-_{\sgn(J_{\varphi_t} \exp(\xi_{\varphi_t} ))}(\log |J_{\varphi_t} \exp\left(\xi_{\varphi_t} \right)|) \ra \\
=& \,\mathbf{E}^{\log |x|,\sgn(x)} \la \1_{\lc \varphi_t < \tau_{(-\infty,0]} \rc} U^-_{\sgn(J_{\varphi_t})}(\xi_{\varphi_t}) \ra.
\end{align*}
Denote the natural enlargement of the filtration induced by $(\xi,J)$ by $(\cG_t)_{t \geq 0}$.
An important fact that we shall shortly use (and is easy to verify) is that for all $t \geq 0$ the time-change $\varphi_t$ is a stopping-time with respect to $(\cG_v)_{v \geq 0}$ which is almost surely finite with respect to $\mathbf{P}^{x,i}$, for all $x\in\mathbb{R}$ and $i\in\{-1,1\}$.  To understand why the latter is true, note that the symmetric and oscillatory behaviour of  ${X}$ in the current regime of $\alpha$ ensures that $\limsup_{t\to\infty}\xi_t = \infty$ almost surely, irrespective of the point of issue of $(\xi, J)$.
This, in turn, means that $\lim_{t\to\infty} \int_0^t{\rm e}^{\alpha \xi_s}\dd s = \infty$ in the same almost sure sense; see for example the discussion of the different properties of stable processes and their underlying MAPs  in Kyprianou \cite{ALEA_kyp}. Consequently, as $(\varphi_t, t\geq 0)$ is the inverse of the aforementioned integrated exponential of $\xi$, it follows that, for each $t\geq 0$, $x\in\mathbb{R}$ and $i\in\{-1,1\}$ we have that $\varphi_t<\infty$, $\mathbf{P}^{x,i}$-almost surely. It now follows from Theorem III.3.4 of Jacod and Shiryaev \cite{JS} that 
\[
\mathbf{E}^{\log |x|,\sgn(x)} \la \1_{\lc \varphi_t < \tau_{(-\infty,0]} \rc} U^-_{\sgn(J_{\varphi_t})}(\xi_{\varphi_t}) \ra =
U_{\sgn(x)}(\log|x|),
\]
and hence
$U^-_{\sgn(x)}(\log |x|)$, $|x|>1$, is a harmonic function for the stable process killed on entering $[-1,1]$.\smallskip

\smallskip

 To finish the proof of Theorem \ref{thm_harmonic} we have to show that
 \begin{align*}
 	h(x)=U^-_{\sgn(x)}(\log |x|),\qquad x\in \mR \backslash [-1,1],
 \end{align*} 
 up to a multiplicative constant, where $h$ is the explicit function from Theorem \ref{thm_harmonic}. Kyprianou et al. \cite{Kyp_Riv_Sen_01}, Corollary 1.6, found explicit densities for $U^-_{i}(\dd x)$. Using these results we have, for $x>1$, that 
\begin{align*}
U^-_{\sgn(x)}(\log |x|) 
&=   \int_0^{\log x} (1-{\rm e}^{-z})^{-\frac{1}{2}} (1+{\rm e}^{-z})^{\frac{1}{2}}  + (1-{\rm e}^{-z})^{\frac{1}{2}} (1+{\rm e}^{-z})^{-\frac{1}{2}} \, \dd z \\
&= 2 \int_1^{x} (u^2-1)^{-\frac{1}{2}}\, \dd u\\
&= 2 h(x).
\end{align*}
By symmetry the claim for $x<-1$ follows analogously.

\subsubsection{$\underline{\alpha \in (1,2)}$}
The idea of the argument is as follows. The multiplication of harmonic functions corresponding to the concatenation of $h$-transforms gives a new harmonic function. In our setting, recall from Section \ref{sec:pre} that $ |x|^{\alpha-1}, x\neq 0,$ is harmonic for $\mP^{\dag,x}$ and the $h$-transformed process delivers the Markov probabilities  $\mP^{\circ,x}$, $x\neq 0$. We will show that $h^\circ(x): =  |x|^{1-\alpha} h(x)$, $x\in\mathbb{R}\backslash [-1,1]$, with $h$ defined in Theorem \ref{thm_harmonic},  is the probability of avoiding $[-1,1]$ for $\mP^{\circ,x}$ ($\mP^{\circ,x}$ is transient) and as a consequence is harmonic for $\mP^{\circ,x}$ killed in $[-1,1]$:
\begin{align*}
	\mP^{x} \stackrel{\text{killing at 0, $h$-transform with }|x|^{\alpha-1}}{\xrightarrow{\hspace*{4.8cm}}}\mP^{\circ,x}\stackrel{\text{killing in [-1,1], $h$-transform with }|x|^{1-\alpha}h(x)}{\xrightarrow{\hspace*{6cm}}} \mP^{\updownarrow,x}.
\end{align*}
From this idea it  turns out that $h(x)=|x|^{\alpha-1}h^\circ(x)$ is harmonic for $\mP^x$ killed on entering $[-1,1]$. Later we will prove that conditioning $\mP^x$ to avoid $[-1,1]$ is nothing but conditioning $\mP^x$ to avoid zero and then to condition $\mP^{\circ,x}$ to avoid $[-1,1]$.
\begin{remark}\rm
Our argument resonates with the work of  Hirano \cite{Hir_01} who looked at the conditioning of a L\'evy process to stay positive. Under the assumption of a positive Cram\'er number, Hirano conditioned a L\'evy process that drifts to $-\infty$ to stay positive by first Esscher transforming, which is equivalent to conditioning to `drift towards $+\infty$', and then conditioning the Esscher transform to stay positive.
\end{remark}
In the following lemma the Riesz--Bogdan--\.Zak transform is used to identify the harmonic function for $\mP^{\circ,x}$ killed in $[-1,1]$.
\begin{lemma}\label{lemma_harmonic_Bog}
Define $g(x) =|x|^{1-\alpha} h(x)$, then
\begin{itemize}
\item[(i)] $ \mP^{\circ,x}(T_{[-1,1]} = +\infty)= (\alpha-1) g(x)$ for all $x \notin [-1,1]$, 
\item[(ii)] the function
 $g$ is harmonic for $({\emph{\texttt P}}_t^{\circ,[-1,1]})_{t \geq 0}$, where
$${\emph{\texttt P}}_t^{\circ,[-1,1]} f(x) \coloneqq \mE^{\circ,x}\la f({X}_t) \1_{\lc  t <T_{[-1,1]} \rc} \ra.$$
\end{itemize}
\end{lemma}
\begin{proof}
(i) Let $\underline{m}$ be the $[0,\infty)$-valued time such that $|{X}_{\underline{m}}| \leq |{X}_t|$ for all $t \geq 0$ (point of closest reach) and $\overline{m}$ the time such that $|{X}_{\overline{m}}| \geq |{X}_t|$ for all $t \leq T_0$ (point of furthest reach). Then we see for $x \notin [-1,1]$ using the Riesz--Bogdan--\.Zak transform in the second equation:
\begin{align}
\mP^{\circ,x}(T_{[-1,1]} = +\infty) &= \mP^{\circ,x}(|{X}_{\underline{m}}|>1)\nonumber\\
&= \hat{\mP}^{\frac{1}{x}}\left(\Big|\frac{1}{{X}_{\overline{m}}}\Big|>1\right) \nonumber\\
&= \hat{\mP}^{\frac{1}{x}}(|{X}_{\overline{m}}|<1)\nonumber\\
&= \hat{\mP}^{\frac{1}{x}}\left({X}_{\overline{m}} \in \Big[\frac{1}{|x|},1\Big)\right) + \hat{\mP}^{\frac{1}{x}}\left({X}_{\overline{m}} \in \Big(-1,-\frac{1}{|x|}\Big]\right).
\label{2probs}
\end{align}
Now we use Proposition 1.2 of Kyprianou et al. \cite{Kyp_Riv_Sen_01} which gives an explicit expression for the distribution of ${X}_{\overline{m}}$ under $\hat{\mP}^{\frac{1}{x}}$. We consider only the case $x>1$, the case $x<-1$ is similar with $\rho$ replaced by $\hat{\rho}$ and $x$ replaced by $|x|$. We start to compute the first summand in \eqref{2probs} as
\begin{align*}
 &\quad\hat{\mP}^{\frac{1}{x}}\left({X}_{\overline{m}} \in \Big[\frac{1}{x},1\Big)\right)\\
=&\, \frac{\alpha-1}{2}  \int_{\frac{1}{x}}^1 z^{-\alpha} \Big[\big(z+\frac{1}{x}\big)^{\alpha\rho}\big(z-\frac{1}{x}\big)^{\alpha\hat{\rho}-1} -(\alpha-1)\frac{1}{x^{\alpha-1}}  \int_1^{zx} \psi_{\alpha\rho}(v) \, \dd v \Big] \, \dd z \\
=&\, \frac{\alpha-1}{2}  \int_{1}^x u^{-\alpha} \Big[(u+1)\psi_{\alpha\rho}(u) -(\alpha-1)  \int_1^{u} \psi_{\alpha\rho}(v) \, \dd v \Big] \, \dd u,
\end{align*}
where for the third equality we have have made the change variable $z = u/x$.
We can similarly compute the second summand in \eqref{2probs} as
\begin{align*}
\hat{\mP}^{\frac{1}{x}}&\left({X}_{\overline{m}} \in \Big(-1,-\frac{1}{x}\Big]\right)=
 \frac{\alpha-1}{2}  \int_{1}^{x} u^{-\alpha}\Big[(u-1)\psi_{\alpha\rho}(u) -(\alpha-1)  \int_1^{u} \psi_{\alpha\rho}(v) \, \dd v \Big] \, \dd u,
\end{align*}
and, adding the two terms together, it follows that
$$\mP^{\circ,x}(T_{[-1,1]} = +\infty) = (\alpha-1) \int_{1}^{x} \Big[z^{1-\alpha} \psi_{\alpha\rho}(z) -(\alpha-1) z^{-\alpha}  \int_1^{z} \psi_{\alpha\rho}(v) \, \dd v \Big] \, \dd z.$$
Integration by parts yields
\begin{align*}
 \int_{1}^{x} \Big[z^{-\alpha}  \int_1^{z} \psi_{\alpha\rho}(v) \, \dd v \Big] \, \dd z
=& \, -\frac{1}{\alpha-1}x^{1-\alpha}  \int_1^{x} \psi_{\alpha\rho}(v) \, \dd v +  \int_1^x \frac{1}{\alpha-1} z^{1-\alpha}  \psi_{\alpha\rho}(z) \, \dd z,
\end{align*}
hence,
$$\mP^{\circ,x}(T_{[-1,1]} = +\infty) = (\alpha-1) x^{1-\alpha} \int_1^{x} \psi_{\alpha\rho}(v) \, \dd v = (\alpha-1) g(x)$$
for $x>1$. \smallskip

(ii) Using (i) the argument is classical and essentially the same as the one given in Section \ref{a<1}. For the sake of brevity we leave the details to the reader.
\end{proof}

With Lemma \ref{lemma_harmonic_Bog} in hand, it is now straight forward to verify the harmonicity of $h$ using the definition of $\mP^{\circ,x}$ as a change of measure. Indeed, we note that 
\begin{align*}
\mE^x \la \1_{\lc t<T_{[-1,1]} \rc} h({X}_t) \ra &= \mE^{\dagger,x} \la \1_{\lc t<T_{[-1,1]} \rc} g({X}_t)  |{X}_t|^{\alpha-1} \ra\\
&= |x|^{\alpha-1} \mE^{\circ,x} \la \1_{\lc t<T_{[-1,1]} \rc} g({X}_t) \ra \\
&=|x|^{\alpha-1}g(x)\\
&= h(x)
\end{align*}
for all $ x\in\mathbb{R}\backslash [-1,1].$ Hence, $h$ is harmonic for the killed process.

\subsection{Proof of Theorem \ref{thm_cond}}\label{a}
To identify the conditioned processes as $h$-transforms we follow a classical argument based on the Markov property. The argument needs two ingredients. First is the 
existence of an asymptotic tail distribution of the kind 
\begin{equation}
\lim\limits_{s \rightarrow \infty} f(s)\,\mP^x(s < T_{[-1,1]}) =h(x).
\label{fandh}
\end{equation} Second is the harmonicity $h$.

\smallskip

For $\alpha<1$ the argument is straightforward with $f=1$  and the explicit formula for $h(x)=\mP^x(T_{[-1,1]}=\infty)$ already derived in Section \ref{a<1}. Much more interesting are the cases $\alpha=1$ and $\alpha>1$. In both cases, the asymptotic tail distributions are given in classical fluctuation theory results due to Blumenthal et al. \cite{Blu_Get_Ray_01} and Port \cite{Por_01}. Unfortunately, in both cases it is unclear if the limit is harmonic (the limiting expression only implies $h$ is excessive). To justify harmonicity  we use the results from the previous section: for $\alpha=1$ the asymptotic tail distribution of Blumenthal et al. is precisely given by \eqref{precise} which we proved to be harmonic using the Lamperti-Kiu representation. For $\alpha>1$ the limit has a particular form that we can identify (using a recent result of Profeta and Simon \cite{Pro_Sim_01}) as our harmonic function $h$ from the previous section.\smallskip

To summarize the approach, our results only give the harmonicity whereas the known fluctuation theory only gives the needed tail asymptotics. Combining both the limiting procedure from Theorem \ref{thm_cond} can be performed and the limit is identified as $h$-transform with $h$ from Theorem \ref{thm_harmonic}.

\subsubsection{$\underline{\alpha \in (0,1)}$} We start with the simplest case $\alpha\in (0,1)$ where we condition on a positive probability event. As seen in the proof of Theorem \ref{thm_harmonic} the probability of this event is the harmonic function up to a multiplicative constant. That is to say, with $h$ defined as in Theorem \ref{thm_harmonic}, we have
$$h(x)=  \frac{\Gamma(1-\alpha)\Gamma(\alpha\hat{\rho})}{2^{1-\alpha}\Gamma(1-\alpha\rho)}\mP^x(T_{[-1,1]}=\infty).$$
It follows immediately that $h$ is bounded. Hence, for $\Lambda \in \cF_t$, we get with the Markov property and dominated convergence
\begin{align*}
\lim\limits_{s \rightarrow \infty} \mP^x(\Lambda \,|\, t+s < T_{[-1,1]}) 
&=\lim_{s\to\infty} \frac{\mE^x\la\1_\Lambda \1_{\lc t < T_{[-1,1]}\rc} \mP^{{X}_t}( T_{[-1,1]} >s)\ra}{\mP^x(T_{[-1,1]} >t+s)}\\
&= \frac{\mE^x\la\1_\Lambda \1_{\lc t < T_{[-1,1]}\rc} \mP^{{X}_t}( T_{[-1,1]} = \infty)\ra}{\mP^x(T_{[-1,1]} = \infty)}\\
&= \mE^x\la\1_\Lambda \1_{\lc t < T_{[-1,1]}\rc} \frac{h({X}_t)}{h(x)} \ra\\
&= \mP^{\updownarrow,x}(\Lambda).
\end{align*}
\subsubsection{$\underline{\alpha =1}$}\label{seca1}
According to Blumenthal et al. \cite{Blu_Get_Ray_01}, Corollary 3, the tail asymptotics of first hitting times are
\begin{align}\label{abc}
	\lim_{s \rightarrow \infty} \mP^x(s<T_{[-1,1]}) \log (s)  = h(x)
\end{align}
with $h$ from \eqref{precise}. In the previous section we proved that $h$ is harmonic.\smallskip

 Using Fatou's Lemma in the second equality and the strong Markov property in the third equality we get, for $\Lambda \in \cF_t$,
\begin{align*}
\mP^{\updownarrow,x}(\Lambda) &= \frac{1}{h(x)} \mE^x \la \1_\Lambda \1_{\lc t<T_{[-1,1]} \rc} \lim\limits_{s \rightarrow \infty} \log (s) \mP^{{X}_t}(s < T_{[-1,1]}) \ra \\
&\leq \liminf\limits_{s \rightarrow \infty} \frac{\log (s)}{h(x)} \mE^x \la \1_\Lambda \1_{\lc t<T_{[-1,1]} \rc} \mP^{{X}_t}(s < T_{[-1,1]})\ra \\
&= \liminf\limits_{s \rightarrow \infty} \frac{\log (s)}{h(x)} \mP^{x}(\Lambda, t+s < T_{[-1,1]})\\
&= \liminf\limits_{s \rightarrow \infty} \frac{\log (s)}{ \log(t+s) \mP^x(t+s < T_{[-1,1]})} \mP^{x}(\Lambda, t+s < T_{[-1,1]})\\
&=  \liminf\limits_{s \rightarrow \infty} \mP^x(\Lambda \, | \, t+s < T_{[-1,1]}).
\end{align*}
Since $\Lambda^{\mathrm{C}} \in \cF_t$ and we can apply the same calculation for $\Lambda^{\mathrm{C}}$ we also get
$$\mP^{\updownarrow,x}(\Lambda^\mathrm{C}) \leq \liminf\limits_{s \rightarrow \infty} \mP^x(\Lambda^\mathrm        {C} \, | \, t+s < T_{[-1,1]})=  1- \limsup\limits_{s \rightarrow \infty} \mP^x(\Lambda \, | \, t+s < T_{[-1,1]}).$$
At this point in the argument,  it is important  that we have already proved harmonicity of $h$, i.e. that $\mP^{\updownarrow,x}$ is a probability measure. Hence, we can write $\mP^{\updownarrow,x}(\Lambda^\mathrm{C})= 1 - \mP^{\updownarrow,x}(\Lambda)$. This leads us to
$$\mP^{\updownarrow,x}(\Lambda) \geq \limsup\limits_{s \rightarrow \infty} \mP^x(\Lambda \, | \, t+s < T_{[-1,1]})$$
and combining both inequalities the claim follows.
\subsubsection{$\underline{\alpha \in (1,2)}$}\label{sec2}
According to Port \cite{Por_01} the tail asymptotics of first hitting times are
\begin{align}\label{asy}
	\lim\limits_{s \rightarrow \infty} s^{1-\frac{1}{\alpha}}\mP^x(s < T_{[-1,1]}) = c_{\alpha\rho} \lim_{y \rightarrow \infty} u^{[-1,1]}(x,y),
\end{align}
for some constant $c_{\alpha\rho}>0$ and $u^{[-1,1]}(x,y)$ is the density of the potential of the stable process killed on entering $[-1,1]$. Profeta and Simon \cite{Pro_Sim_01} derived explicit formulas for $u^{[-1,1]}(x,y)$, namely
\begin{align}
	u^{[-1,1]}(x,y) &= c_{\alpha\rho}(y-x)^{\alpha-1} \Big(  \int_1^{z(x,y)} \psi_{\alpha\rho}(v) \, \dd v\notag\\
	&\quad -(\alpha-1)(y-x)^{1-\alpha} \int_1^y \psi_{\alpha\hat{\rho}}(v) \, \dd v  \int_1^x \psi_{\alpha\rho}(v) \, \dd v \Big),
	\label{PS}
\end{align}
$1<x<y$.  In the formula we used the abbreviations $z(x,y) = {(xy-1)}/{(y-x)}$ and $c_{\alpha\rho}={2^{1-\alpha}}/{(\Gamma(\alpha\rho)\Gamma(\alpha\hat{\rho}))}$.

 \smallskip

In order to combine Port's asymptotic formula \eqref{asy} with our Theorem \ref{thm_harmonic} we need to send $y$ to infinity in Profeta's and Simon's formula \eqref{PS}.
\begin{lemma}\label{la}
If  $h$ is defined as in Theorem \ref{thm_harmonic}, then, up to a multiplicative constant,
	\begin{align*}
		\lim_{y\to \infty} u^{[-1,1]}(x,y) =h(x),\qquad x\notin [-1,1].
	\end{align*}
\end{lemma}
\begin{proof}
 First note that 
 \begin{align*}
 	\lim_{y\to\infty}  \int_1^{z(x,y)} \psi_{\alpha\rho}(v) \, \dd v= \int_1^{x} \psi_{\alpha\rho}(v) \, \dd v.
\end{align*}	
	 On the other hand we see with l'Hopital's rule
$$\lim_{y \rightarrow \infty} (\alpha-1)(y-x)^{1-\alpha}  \int_1^y \psi_{\alpha\hat{\rho}}(v) \, \dd v = \lim_{y \rightarrow \infty} (y-x)^{2-\alpha} \psi_{\alpha\hat{\rho}}(y) = 1.$$
Hence, $u^{[-1,1]}(x,y)$ is a product of two functions, one tending to $+\infty$, the other tending to $0$. Applying  l'Hopital's rule again to the whole term we get
\begin{align*}
&\,\lim_{y \rightarrow \infty} c_{\alpha\rho}^{-1} u^{[-1,1]}(x,y) \\
=&\, \lim_{y \rightarrow \infty} \frac{1}{(1-\alpha)(y-x)^{-\alpha}} \Bigg[ \frac{1-x^2}{(y-x)^2} \psi_{\alpha\rho}(z(x,y)) \\
&\quad - (\alpha-1)  \int_1^x \psi_{\alpha\rho}(v) \, \dd v \Bigg((1-\alpha)(y-x)^{-\alpha} \int_1^y \psi_{\alpha\hat{\rho}}(v)\, \dd v + \psi_{\alpha\hat{\rho}}(y)(y-x)^{1-\alpha} \Bigg) \Bigg] \\
=&\, \lim_{y \rightarrow \infty}\frac{(1-x^2)\psi_{\alpha\rho}(z(x,y))}{(1-\alpha)(y-x)^{2-\alpha}}  +  \int_1^x \psi_{\alpha\rho}(v) \, \dd v \Bigg((1-\alpha) \int_1^y \psi_{\alpha\hat{\rho}}(v)\, \dd v + \psi_{\alpha\hat{\rho}}(y)(y-x) \Bigg).
\end{align*}
The first summand converges to $0$ since $z(x,y)$ converges to $x$ and $\alpha <2$. So it remains to show that
$$\psi_{\alpha\hat{\rho}}(y)(y-x)- (\alpha-1) \int_1^y \psi_{\alpha\hat{\rho}}(v)\, \dd v$$
converges to some positive constant which is independet of $x$. It is clear that both terms tend to $+\infty$ with order $\alpha-1$. We rewrite the term in the following way:
\begin{align*}
&\quad\psi_{\alpha\hat{\rho}}(y)(y-x)- (\alpha-1) \int_1^y \psi_{\alpha\hat{\rho}}(v)\, \dd v \\
&=\, \psi_{\alpha\hat{\rho}}(y)(y-1) - \psi_{\alpha\hat{\rho}}(y)(x-1) - (\alpha-1) \int_1^y \psi_{\alpha\hat{\rho}}(v)\, \dd v.
\end{align*}
Of course $\psi_{\alpha\hat{\rho}}(y)(x-1)$ converges to $0$. So it remains to show that 
$$\psi_{\alpha\hat{\rho}}(y)(y-1)- (\alpha-1) \int_1^y \psi_{\alpha\hat{\rho}}(v)\, \dd v$$
converges. Note that we can write
$$\psi_{\alpha\hat{\rho}}^\prime(y) = \psi_{\alpha\hat{\rho}}(y) \big[(\alpha\rho-1)(v-1)^{-1}+(\alpha\hat{\rho}-1)(v+1)^{-1}\big],$$
and 
with this we see that
\begin{align*}
&\quad\psi_{\alpha\hat{\rho}}(y)(y-1)- (\alpha-1) \int_1^y \psi_{\alpha\hat{\rho}}(v)\, \dd v \\
&= \,  \int_1^y \left(\psi_{\alpha\hat{\rho}}^\prime(v)(v-1) + \psi_{\alpha\hat{\rho}}(v)\right)\, \dd v - (\alpha-1) \int_1^y \psi_{\alpha\hat{\rho}}(v)\, \dd v \\
&= \,  \int_1^y \psi_{\alpha\hat{\rho}}(v)\left((\alpha\rho-1)+ (\alpha\hat{\rho}-1)\frac{v-1}{v+1} + 1 -(\alpha-1)\right) \, \dd v \\
&= \, 2(1-\alpha\hat{\rho})  \int_1^y \psi_{\alpha\hat{\rho}}(v) \frac{1}{v+1} \, \dd v.
\end{align*}
Since  $\psi_{\alpha\hat{\rho}}(v) /(v+1)$ behaves like $(v-1)^{\alpha\rho-1}$ for $v \searrow 0$ and like $v^{\alpha-3}$ for $v \rightarrow +\infty$, it follows that $ \int_1^\infty \psi_{\alpha\hat{\rho}}(v)/(v+1) \, \dd v \in (0,\infty)$ because $\alpha\rho \in (0,1)$ and $\alpha-3 <-1$. The claim for $x<-1$ follows by applying duality and the formula for $u^{[-1,1]}(x,y)$ for $x>1$ and $y<-1$, similar to \eqref{PS}, which is also found in  Profeta and Simon \cite{Pro_Sim_01}.
\end{proof}
The proof of Theorem \ref{thm_cond} in this regime of $\alpha$ can now be copied from Section \ref{seca1} replacing $\log (s)$ by $s^{1-\frac{1}{\alpha}}$ in \eqref{abc}. The desired harmonicity of  $\lim\limits_{s \rightarrow \infty} s^{1-\frac{1}{\alpha}}\mP^x(s < T_{[-1,1]}) ,$ as a function of $x$, comes from Lemma \ref{la} combined with Theorem \ref{thm_harmonic}.


\subsection{Proof of Theorem \ref{thm_trans}} We appeal to a very classical result of Getoor, which seems not to have been used before in analogous scenarios. We mention for example proving transience in the setting of  L\'evy process conditioned to be positive (see e.g. the approach in  Lemma VII.12 of Bertoin \cite{Bert_01}), where the following proof would work equally well. 

\smallskip

As our conditioned stable process conforms to the definition of a regular Markov process according to the definition of Getoor \cite{G}, it follows from the same paper that the process is transient if there exists a sequence of domains $B_n$, $n\in\mathbb{N}$, such that $B_n\uparrow\mathbb{R}\backslash[-1,1]$ such that  

\[
	\mE^{\updownarrow,x} \Bigg[  \int_0^\infty \1_{\lc {X}_t \in B_n \rc} \, \dd t \Bigg] < \infty,\qquad |x|>1,
\]
for each $n\in\mathbb{N}$. Here, Getoor's definition of transience implies that, for all bounded intervals $B$ in $\mathbb{R}$, $\mP^{\updownarrow,x}(\ell_B<\infty)=1 $, for all $|x|>1$, where
\[
\ell_B = \sup\{t>0: {X}_t\in B\}.
\]
Note, in turn, this implies the notion of transience highlighted in the statement of Theorem \ref{thm_trans}.

\smallskip

It therefore suffices to prove that 
\begin{align}\label{newuj}
	\mE^{\updownarrow,x} \Bigg[  \int_0^\infty \1_{\lc {X}_t \in [-d,d] \rc} \, \dd t \Bigg] < \infty,\qquad |x|>1,
\end{align}
for all $d>1$.
To this end,  note that
\begin{align*}
\mE^{\updownarrow,x} \Bigg[  \int_0^\infty \1_{\lc {X}_t \in [-d,d] \rc} \, \dd t \Bigg] &=  \int_0^\infty \mE^{x}\la \1_{\lc {X}_t \in [-d,d] \rc} \1_{\lc t < T_{[-1,1]} \rc} \frac{h({X}_t)}{h(x)} \ra \, \dd t  \\
&\leq \sup_{y \in [-d,d]\backslash [-1,1]} \frac{h(y)}{h(x)} \, \mE^x \Bigg[  \int_0^{T_{[-1,1]}} \1_{\lc {X}_t \in [-d,d] \rc} \, \dd t \Bigg] \\
&= \sup_{y \in [-d,d]\backslash [-1,1]} \frac{h(y)}{h(x)} \, U_{[-1,1]}\big(x,[-d,d]\big),
\end{align*}
where $U_{[-1,1]}\big(x,\dd y\big)$ is the potential measure of the stable process killed on entering $[-1,1]$. The explicit form of $h$ implies
$$ \sup_{y \in [-d,d]\backslash [-1,1]} h(y) = \max(h(-d),h(d)) <\infty.$$
Hence, it is sufficient to show finiteness of the potentials $U_{[-1,1]}(x,[-d,d])$, $|x|>1$. Taking account of the explicit formula for the killed potential density $u^{[-1,1]}(x,y)$ from Profeta and Simon \cite{Pro_Sim_01}, it is a straightforward, if not tedious, analytical exercise to verify that indeed $U_{[-1,1]}(x,[-d,d]) = \int_{[-d,d]\backslash[-1,1]}u^{[-1,1]}(x,y) \,\dd y$ is finite.

\bibliographystyle{abbrvnat}
\bibliography{references}

\end{document}